\documentclass[12pt]{amsart}
\usepackage{amsmath,amsfonts,amsthm,amssymb,mathrsfs,amsxtra,amscd,latexsym, xcolor,float}
\usepackage[hmargin=2cm,vmargin=3cm]{geometry}
\usepackage{hyperref}
\usepackage{hyperref}
\hypersetup{
  colorlinks   = true, 
  urlcolor     = blue, 
  linkcolor    = blue, 
  citecolor   = red 
}

\usepackage[OT2,T1]{fontenc}

\usepackage{amsmath,amsthm,amssymb,enumerate,amsfonts,mathrsfs,amsxtra,amscd,latexsym,graphicx}

\newtheorem{theorem}{Theorem}[section]
\newtheorem*{theorem*}{Theorem}

\newtheorem{lemma}[theorem]{Lemma}

\newtheorem*{conjecture}{Conjecture}

\newtheorem*{remark}{Remark}

\numberwithin{equation}{section}

\DeclareSymbolFont{cyrletters}{OT2}{wncyr}{m}{n}\DeclareMathSymbol{\Sha}{\mathalpha}{cyrletters}{"58}

\newcommand{\sm}{\left(\begin{smallmatrix}}
\newcommand{\esm}{\end{smallmatrix}\right)}
\newcommand{\mat}{\left(\begin{matrix}}
\newcommand{\emat}{\end{matrix}\right)}

\def\det{\mathrm{det}}

\def\Tr{\mathop{\rm Tr}}

\def\re{\mathrm{Re}}

\def\Gal{\mathrm{Gal}}

\def\Frob{\mathrm{Frob}}
\newcommand{\Tam}{\operatorname{Tam}}
\newcommand{\Directsum}{\bigoplus}

\def\GL{\mathrm{GL}}

\def\SL{\mathrm{SL}}

\def\Q{\mathbb Q} \def\R{\mathbb R} \def\Z{\mathbb Z} \def\C{\mathbb C}

\newcommand{\beq}{\begin{eqnarray*}}
\newcommand{\eeq}{\end{eqnarray*}}
\newcommand{\beqn}{\begin{eqnarray}}
\newcommand{\eeqn}{\end{eqnarray}}

\newcommand{\ben}{\begin{enumerate}}
\newcommand{\een}{\end{enumerate}}

\begin{document}

\title{Zeta-polynomials for modular form periods}

\author{}

\address{}
\email{}
\author{Ken Ono}

\address{Department of Mathematics and Computer Science, Emory University, Atlanta, Georgia 30322}
\email{ono@mathcs.emory.edu}

\author{Larry Rolen}

\address{Department of Mathematics, Penn State University, University Park, Pennsylvania 16802}
\email{larryrolen@psu.edu}

\author{Florian Sprung}
\address{Department of Mathematics, Princeton University, Princeton, New Jersey 08544}
\email{fsprung@princeton.edu}

\begin{abstract}  Answering problems of Manin, we use the critical $L$-values of
even weight $k\geq 4$ newforms $f\in S_k(\Gamma_0(N))$  to define zeta-polynomials $Z_f(s)$
which satisfy the functional
equation $Z_f(s)=\pm Z_f(1-s)$, and which obey the Riemann Hypothesis: if $Z_f(\rho)=0$, then $\re(\rho)=1/2$. The zeros of the $Z_f(s)$ on the critical line in $t$-aspect are distributed in a manner which is somewhat analogous to those of classical zeta-functions. 
These polynomials are assembled using (signed) Stirling numbers and  ``weighted moments'' of critical $L$-values. In analogy with Ehrhart polynomials which keep track of integer points in polytopes, the $Z_f(s)$ keep track of
arithmetic information.
Assuming the Bloch--Kato Tamagawa Number Conjecture, they encode  the arithmetic of a combinatorial arithmetic-geometric object which we call the ``Bloch-Kato complex'' for $f$. Loosely speaking, these are graded sums of weighted moments of orders of \v{S}afarevi\v{c}--Tate groups associated to the Tate twists of the modular motives.
\end{abstract}

\thanks{The first author thanks the support of the Asa Griggs Candler Fund and the NSF. The authors thank Masataka Chida, Yuri Manin and Don Zagier for useful discussions and correspondence. This material is based upon work supported by the National Science Foundation under grants DMS-0964844, DMS-1601306 and DMS-1128155. Any opinions, findings and conclusions or recommendations expressed in this material are those of the authors and do not necessarily reflect the views of the National Science Foundation.}

\maketitle

\section{Introduction and Statement of Results} \label{section1}

\noindent Let $f \in S_k(\Gamma_0(N))$ be a newform  of even weight $k$ and level $N$.  Associated to $f$ is its $L$-function $L(f,s)$, which may be normalized so that the completed $L$-function 
$$
\Lambda(f,s) := \Big(\frac{\sqrt{N}}{2\pi}\Big)^s \Gamma(s) L(f,s),   
$$ 
satisfies the functional equation $\Lambda(f,s) = \epsilon(f) \Lambda(f,k-s)$, with $\epsilon(f)= \pm 1$.  The {\it critical $L$-values} are the complex numbers
$L(f,1)$, $L(f,2)$, $\dots$, $L(f,k-1)$.  

In a recent paper \cite{M}, Manin speculated on the existence of natural {\it zeta-polynomials} which can be canonically assembled
from these critical values.  A polynomial $Z(s)$ is a {\it zeta-polynomial}  if it is arithmetic-geometric in origin, satisfies
a functional equation of the form
$$Z(s)=\pm Z(1-s)$$ and obeys the Riemann Hypothesis: 
if $Z(\rho)=0$, then $\re(\rho)=1/2$.

Here we confirm his speculation.
To this end, we define the $m$-th {\it weighted moments} of critical values
\begin{equation}\label{Moments}
M_f(m):=\sum_{j=0}^{k-2}\left(\frac{\sqrt N}{2\pi}\right)^{j+1}\frac{L(f,j+1)}{(k-2-j)!}j^m=\frac1{(k-2)!}\sum_{j=0}^{k-2}\binom{k-2}j\Lambda(f,j+1)j^m.
\end{equation}
For positive integers $n$, we recall the usual generating function for
the (signed) Stirling numbers of the first kind
\begin{equation}
(x)_n=x(x-1)(x-2)\cdots (x-n+1)=:\sum_{m=0}^n s(n,m)x^m.
\end{equation}
Using these numbers we define
the zeta-polynomial for these weighted moments by 
\begin{equation}\label{Zdef}
Z_f(s):=\epsilon(f)\cdot\sum_{h=0}^{k-2} (-s)^h \sum_{m=0}^{k-2-h} \binom{m+h}{h}\cdot s(k-2,m+h)\cdot M_f(m).
\end{equation}

To be a zeta-polynomial in the sense of Manin \cite{M}, we must show that $Z_f(s)$ satisfies a functional equation and the Riemann Hypothesis.
Our first result confirms these properties.

\begin{theorem}\label{Thm1} If $f\in S_k(\Gamma_0(N))$ is an even weight $k\geq 4$ newform, then the following are true:
\begin{enumerate}
\item For all $s\in \C$ we have that $Z_f(s)=\epsilon(f)Z_f(1-s)$.
\item If $Z_f(\rho)=0$, then $\re(\rho)=1/2$.
\end{enumerate}
\end{theorem}

It is natural to study the distribution of the zeros of $Z_f(s)$ on the line $\re(s)=1/2$. 
Although the $Z_f(s)$ are polynomials, do their zeros behave in a manner which is analogous to
the zeros of the Riemann zeta-function $\zeta(s)$? Namely, how are their zeros distributed in comparison with the growth of
$$N(T):= \# \{  \rho=s+it\ : \ \zeta(\rho)=\frac12 \ \ {\text {\rm with}}\ 0<t\leq T\}, 
$$
which is well known to satisfy
\begin{equation}\label{NT}
N(T)=\frac{T}{2\pi} \log\frac{T}{2\pi} -\frac{T}{2\pi}+O(\log T)?
\end{equation}
We find that the zeros of $Z_f(s)$ 
behave in a manner that is somewhat analogous to (\ref{NT}) in terms of its highest zero.

To make this precise, we find it useful to compare the $Z_f(s)$ with two families of combinatorial polynomials.
In what follows, we note that for $x,y\in\C$, the binomial coefficient $\binom xy$ is defined by 
\[
\binom{x}{y}:=\frac{\Gamma(x+1)}{\Gamma(y+1)\Gamma(x-y+1)}.
\]
We find that the $Z_f(s)$, depending on $\epsilon(f)$, can naturally be compared with the polynomials
\begin{equation}
H_{k}^{+}(s):=\binom{s+k-2}{k-2}+\binom{s}{k-2}
,
\end{equation}
\begin{equation}
H_{k}^{-}(s):=\sum_{j=0}^{k-3}\binom{s-j+k-3}{k-3}.
\end{equation}

\begin{theorem}\label{ZeroDistribution}
Assuming the notation and hypotheses above, the following are true: 
\begin{enumerate}
\item The zeros of $H^{-}_{k}(-s)$ lie on the line $\re(s)=1/2$,  and they are the complex numbers $\rho=\frac{1}{2}+it$ 
where the $t$ are the real numbers such that the value of the monotonically decreasing function 
\[
h_k(t):=\sum_{j=0}^{k-3}\cot^{-1}\left(\frac{2t}{2j+1}\right)
\]
lies in the set $\{\pi,2\pi,\ldots,(k-3)\pi \}$. Similarly, the zeros of $H^{+}_{k}(-s)$ lie on the line $\re(s)=1/2$ and have imaginary parts $t$ which may be found by solving for $h_k(t)$ to lie in the set $\{\pi/2,3\pi/2,\ldots,(k-5/2)\pi \}$.
Moreover, as $k\rightarrow\infty$, the highest pair of complex conjugate roots (i.e., those whose imaginary parts have the largest absolute values) of $H^{-}_k(s)$ have imaginary part equal in absolute value to
\[
\frac{(k-3)(k-1)}{2\pi}+O(1)
,
\]
and the height of the highest roots of $H^+_k(s)$ is 
\[
\frac{(k-3)(k-1)}{\pi}+O(1)
.
\]
\item Let $f\in S_4(\Gamma_0(N))$ be a newform. If $\epsilon(f)=-1$, then the only root of $Z_f(s)$ is at $s=1/2$. If $\epsilon(f)=1$, then there are two roots of $Z_f(s)$, and as $N\rightarrow\infty$, their roots converge on the sixth order roots of unity $\operatorname{exp}(\pm\pi i /3)$. 
\item For fixed $k\geq6$, as $N\rightarrow +\infty$, the zeros of $Z_f(s)$ for newforms $f\in S_k(\Gamma_0(N))$ with $\epsilon(f)=\pm 1$ converge to the zeros of $H^{\pm}_{k}(-s)$. 
Moreover, for all $k,N$, if $\epsilon(f)=1$ {\text {\rm (resp. $\epsilon(f)=-1$)}}, then the imaginary part of the largest root is strictly bounded by 
$(k-3)\left(k-\frac 72\right)$  {\text {\rm (resp. 
$(k-4)\left(k-\frac 92\right)$)}}.
\end{enumerate}
\end{theorem}

\begin{remark}
Theorem~\ref{ZeroDistribution} (2) is somewhat analogous to (\ref{NT}). Since the zeros of $Z_f(s)$ are approximated by those of  $H^{\pm}_{k}(-s)$, the analog of $N(T)$ is dictated by Theorem~\ref{ZeroDistribution} (1), where the largest zero
has imaginary part $\sim \frac{k^2}{2\pi}$ or $\sim  \frac{k^2}{\pi}$ depending on the sign of the functional equation. 
\end{remark}

By means of the ``Rodriguez-Villegas Transform''  of  \cite{RV}, Theorem~\ref{Thm1}  is naturally related to the arithmetic of period polynomials\footnote{This is a slight reformulation of the period polynomials considered in
references such as \cite{CPZ, KZ, PP, Z}.}
\begin{equation}
R_f(z):=\sum_{j=0}^{k-2} \binom{k-2}{j}\cdot \Lambda(f,k-1-j)\cdot z^j.
\end{equation}
The values of $Z_f(s)$ at non-positive integers are the coefficients expanded around $z=0$ of the rational function $$\frac{R_f(z)}{(1-z)^{k-1}}.$$
This result, which we state next, can be thought of as a natural analogue of the well-known exponential generating function of the values of the Riemann zeta function at negative integers:
\begin{equation}\label{BN}
\frac{t}{e^t-1}=1-\frac12 t+t\sum_{n=1}^{\infty}\zeta(-n)\cdot \frac{(-t)^n}{n!}
.
\end{equation}
\begin{theorem}\label{Thm2} Assuming the notation and hypotheses above, as a power series in $z$ we have
$$
\frac{R_f(z)}{(1-z)^{k-1}}=\sum_{n=0}^{\infty} Z_f(-n)z^n.
$$
\end{theorem}

\begin{remark}
The generating function (\ref{BN}) for the values $\zeta(-n)$ has a well-known interpretation in $K$-theory \cite{FG}. It is essentially the generating function
    for the torsion of the $K$-groups for $\Q$. In view of this interpretation, it is natural to ask whether the $z$-series in Theorem \ref{Thm2} 
    has an analogous interpretation. In other words, what (if any) arithmetic information is encoded by the values $Z_f(-n)$?
Manin recently speculated \cite{M}  on the existence of results such as Theorems~\ref{Thm1} and \ref{Thm2}.
Indeed, in \cite{M} he produced similar zeta-polynomials by applying the Rodriguez-Villegas transform \cite{RV} to the odd period polynomials for Hecke eigenforms on $\SL_2(\Z)$ studied by
Conrey, Farmer, and Imamo$\mathrm{\bar{g}}$lu \cite{CFI}. He asked for a generalization for the full period polynomials
for such Hecke eigenforms in connection to recent work of El-Guindy and Raji \cite{ER}. Theorems \ref{Thm1} and \ref{Thm2} answer this question and
provide the generalization 
for all even weight $k\geq 4$ newforms on congruence subgroups of the form $\Gamma_0(N)$. Theorem~\ref{Thm1} additionally offers an explicit combinatorial description of the zeta-polynomials in terms of weighted moments.
\end{remark}

We offer a conjectural combinatorial arithmetic-geometric interpretation of the $Z_f(s)$.
To this end, we make use of the Bloch-Kato Conjecture, which offers a Galois cohomological interpretation for critical values of motivic $L$-functions
\cite{BK}. Here we consider the special case of the critical values $L(f,1), L(f,2),\dots, L(f,k-1)$.
These conjectures are concerned with motives $\mathcal{M}_f$ associated to $f$, but the data needed for this conjecture can be found in the {\it $\lambda$-adic realization} $V_\lambda$ of $\mathcal{M}_f$ for a prime $\lambda$ of $\Q(f)$, where $\Q(f)$ is the field generated by the Hecke eigenvalues $a_n(f)$ (where we have $a_1(f)=1$). The Galois representation $V_\lambda$ associated to $f$ is due to Deligne, and we recall the essential properties below. For a high-brow construction of $V_\lambda$ from $\mathcal{M}_f$, we refer to the seminal paper of Scholl \cite{scholl}.

Deligne's theorem \cite{D} says that for a prime $\lambda$ of $O_{\Q(f)}$ lying above $l$, there is a continuous linear representation $V_\lambda$ unramified outside $lN$
$$ \rho_{f,\lambda}: \Gal(\overline{\Q}/\Q)\rightarrow \GL(V_\lambda)$$
 so that for a prime $p \nmid lN$, the arithmetic Frobenius $\Frob_p$ satisfies 
 $$\Tr(\rho_f(\Frob_p^{-1}))=a_p(f), \text{ and } \det(\rho_f(\Frob_p^{-1}))=p^{k-1}.$$
 
We may also consider the $j$-th Tate twist $V_\lambda(j)$, which is $V_\lambda$ but with the action of Frobenius multiplied by $p^j$.  After choosing a $ \Gal(\overline{\Q}/\Q)$-stable lattice $T_\lambda$ in $V_\lambda$, we may consider the short exact sequence $$ 0 \longrightarrow T_\lambda(j) \longrightarrow V_\lambda(j)\stackrel{\pi}\longrightarrow V_\lambda/T_\lambda(j)\longrightarrow 0.$$
Bloch and Kato define local conditions $H^1_\mathbf{f}(\Q_p,V_\lambda(j))$ for each prime $p$, discussed in more detail in Section 3. We let $H^1_\mathbf{f}(\Q,V_\lambda(j))$ be the corresponding global object, i.e. the elements of $H^1(\Q,V_\lambda(j))$ whose restriction at $p$ lies in $H^1_\mathbf{f}(\Q_p,V_\lambda(j))$. Analogously, we may define $H^1_\mathbf{f}(\Q,V_\lambda/T_\lambda(j))$, which is the Bloch--Kato $\lambda$- Selmer group. The \v{S}afarevi\v{c}--Tate group is
$$\Sha_f(j)=\Directsum_{\lambda|l} \frac{H^1_\mathbf{f}(\Q,V_\lambda/T_\lambda(j))}{\pi_*H^1_\mathbf{f}(\Q,V_\lambda(j))}.$$
The Bloch--Kato Tamagawa number conjecture then asks the following:

\begin{conjecture}[Bloch--Kato]\label{bkt} Let $0\leq j\leq k-2$, and assume $L(f,j+1)\neq0$. Then we have
$$\frac{L(f,j+1)}{(2\pi i)^{j+1} \Omega^{(-1)^{j+1}}}=u_{j+1} \times \frac{\Tam(j+1) \# \Sha_f(j+1)}{\#H^0_\Q(j+1)\#H^0_\Q(k-1-j)}=:C(j+1) $$
\end{conjecture}

Here, $\Omega^\pm$ denotes the Deligne period, $\Tam$ the product of the Tamagawa numbers, $H^0_\Q$ is the set of global points (precisely defined in Section $3$), and $u_{j+1}$ is a non-specified unit of $\Q(f)$.

\begin{remark}
Note that $L(f,j+1)\neq0$ in this range provided that $j+1\neq k/2$. 
\end{remark}
We denote the normalized version of $C(j+1)$ by 
\begin{equation}
\widetilde{C(j+1)}=C(j+1)\cdot \frac{(i\sqrt{N})^{j+1}\Omega^{(-1)^{j+1}}}{(k-2-j)!},
\end{equation}
but when $L(f,j+1)=0$, we define $\widetilde{C(j+1)}:=0$. 

\begin{theorem}\label{Thm3} Assuming the Bloch-Kato Conjecture and the notation  above, we have that
 $$
 M_f(m)=\sum_{0\leq j\leq k-2} \widetilde{C(j+1)}j^m,
 $$
 which in turn implies for each non-negative integer $n$ that
 $$
 Z_f(-n)= \epsilon(f)\sum_{j = 0}^{k-2}\left(\sum_{h=0}^{ k-2}\sum_{m = 0}^{k-2-h} n^h  \binom{m+h}{h}\cdot s(k-2,m+h)\right)j^m\widetilde{C(j+1)}.
 $$
\end{theorem}

Each $Z_f(s)$ can be thought of as an arithmetic-geometric Ehrhart polynomial, and the combinatorial structure in Theorem~\ref{Thm3}, which we call
the ``Bloch-Kato complex'', serves as an analogue of a polytope. 
Assuming the Bloch-Kato Conjecture, Theorem~\ref{Thm3} describes the values $Z_f(-n)$ as combinatorial sums of $m$-weighted moments of the $j$-th Bloch-Kato components. 
To describe this combinatorial structure, we made use of  the Stirling numbers $s(n,k)$ which can be arranged in a ``Pascal-type triangle'' 
$$
\begin{array}{ccccccccccccc}
 &  &   &        &      &       & 1 &  &   &  &  &  \\
& &   &       &      &  0     &  & 1 &   &  &  &   &  \\
&  &  &        & 0    &       & -1 & & 1 &  & &  & \\
& &  & 0      &     & 2       &    &  -3    &  & 1 &  &  &   \\
&  & 0 &       & -6  &      & 11  &         & -6 &  & 1 &  &  \\
  & 0 &   & 24    &    &-50     &    & 35      &   & -10 &   & 1 &  \\
0 &    & -120 &   & 274    &    & -225   &  & 85 &  & -15 &  & 1\\
\end{array}
$$
thanks to the recurrence relation
$$
s(n,k)=s(n-1,k-1)-(n-1)\cdot s(n-1,k).
$$
This follows from the obvious relation
$$
(x)_n=(x)_{n-1} (x-n+1)=x(x)_{n-1}-(n-1)(x)_{n-1}.
$$
The Bloch-Kato complex is then obtained by cobbling together weighted layers of these Pascal-type triangles using the binomial coefficients appearing in (\ref{Zdef}).

The connection to Ehrhart polynomials arises from the central role of  the $H_k^{\pm}(-s)$ in our study of the $Z_f(s)$.
In \cite{RV}, Rodriguez-Villegas proved that certain  Hilbert polynomials, such as the  $H^{\pm}_{k}(-s)$, which are Rodriguez-Villegas transforms of $x^{k-2}\pm1$, are examples of zeta-polynomials. 
These well-studied combinatorial polynomials  encode important geometric structure such as the distribution of integral points in polytopes.

Given a $d$-dimensional integral lattice polytope $\mathcal P$ in $\R^n$, we recall that the Ehrhart polynomial $\mathcal L_\mathcal P(x)$ is determined by 
\[
\mathcal L_\mathcal P(m)=\#\left\{p\in\Z^n : p\in m\mathcal P\right\}
.
\]
The polynomials $H_{k}^-(s)$ whose behavior determines an estimate for those of $Z_f(s)$ (when $\epsilon(f)=-1)$  as per Theorem \ref{ZeroDistribution} are the Ehrhart polynomials of the simplex (cf. \cite{BHW})
 \[\operatorname{conv}\left\{e_1,e_2,\ldots,e_{k-3},-\sum_{j=1}^{k-3}e_j\right\},
\]
where $e_i$ denotes the $i$-th unit vector in $\R^{k-3}$.
We note that in Section 1.10 of \cite{GRV}, Gunnells and Rodriguez-Villegas also gave an enticing interpretation of the modular-type behavior of Ehrhart polynomials. Namely, they noted that the polytopes $P$ with vertices in a lattice $L$, when acted upon by $\operatorname{GL}(L)$ in the usual way, have a fixed Ehrhart polynomial for each equivalence class of polytopes. Hence, these classes may be thought of as points on a ``modular curve'', and the $\ell$-th coefficient of the Ehrhart polynomial is analogous to a weight $\ell$ modular form. This analogy is strengthened as they define a natural Hecke operator on the set of Ehrhart polynomials, such that the $\ell$-th coefficients of them are eigenfunctions. Moreover, they show that these eigenclasses are all related to explicit, simple Galois representations. Thus, it is natural, and intriguing, to speculate on the relationship between these observations and our Theorem \ref{ZeroDistribution}. In particular, we have shown that as the level $N$ of cusp forms of a fixed weight $k$ tends to infinity, the coefficients of the zeta-polynomial $Z_f(s)$ tend to (a multiple of) these coefficients of Ehrhart polynomials considered in \cite{GRV}. It is also interesting to note that Zagier defined \cite{ZagHecke} a natural Hecke operator on the period polynomials of cusp forms, which commutes with the usual Hecke operators acting on cusp forms, and so one may ask if there is a reasonable interpretation of Hecke operators on the zeta functions $Z_f(s)$ which ties together this circle of ideas.

Now we describe the organization of this paper. In Section~\ref{Zeta} we prove Theorems~\ref{Thm1}, \ref{ZeroDistribution}, and \ref{Thm2}. We make use of recent work of
Jin, Ma, Soundararajan and the first author \cite{JinMaOnoSound} on zeros of period polynomials for modular forms,
the framework of Rodriguez-Villegas transforms \cite{RV}, and 
results of Bey, Henk, and Wills \cite{BHW} on the polynomials $H_k^{\pm}(s)$.
In Section~\ref{BlochKato} we briefly recall the Bloch-Kato Conjecture for the critical values of modular $L$-functions, and give the proof
of Theorem~\ref{Thm3}. We conclude in Section \ref{ExSec} with two examples. 

\section{Proof of Theorems ~\ref{Thm1}, \ref{ZeroDistribution},  and \ref{Thm2}}\label{Zeta}
Here we prove Theorems ~\ref{Thm1}, \ref{ZeroDistribution}, and \ref{Thm2}. We begin by recalling key results of Rodriguez-Villegas. 

\subsection{Theorem of Rodriguez-Villegas}
Here we recall important observations which were cleverly assembled in \cite{RV}. We provide a special case of these results which is most convenient for our purposes. First suppose that $U(z)$ is a polynomial of degree $e$ with $U(1)\neq0$. Then consider the rational function
\[
P(z):=\frac{U(z)}{(1-z)^{e+1}}
.
\]
Expanding as a power series in $z$, we have
\[
P(z)=\sum_{n=0}^{\infty}h_nz^n
,
\]
and it is easily shown that there is a polynomial $H(z)$ of degree $e$ such that for each $n$ we have $H(n)=h_n$. We then have the following ``zeta-like'' properties for the function $Z(s):=H(-s)$. 
\begin{theorem}[Rodriguez-Villegas]\label{RVThm}
If all roots of $U$ lie on the unit circle, then all roots of $Z(s)$ lie on the vertical line $\operatorname{Re}(z)=1/2$. Moreover, if $U$ has real coefficients and $U(1)\neq0$, then $Z(s)$ satisfies the functional equation
\[
Z(1-s)=(-1)^{e}Z(s)
.
\]
\end{theorem}
\begin{proof}
The first claim is simply the special case of the Theorem of \cite{RV} when $d=e+1$. The second claim was described in Section 4 of \cite{RV}, but for the reader's convenience we sketch the proof. By the single proposition of \cite{RV}, it suffices to show that $P(1/z)=(-1)^{e+1}zP(z)$. Now suppose that $U$ factors as
\[
U(z)=(z-\rho_1)\ldots(z-\rho_e)
,
\]
where each $\rho_j$ is on the unit circle but not equal to $1$. Then 
\[
z^eU\left(\frac 1z\right)=(1-z\rho_1)\ldots(1-z\rho_e).
\]
Since the coefficients of $U$ are real, we have $(-1)^e\rho_1\rho_2\ldots\rho_e=1$, and dividing by this quantity yields 
\[
U\left(\frac 1z\right)=z^{-e}U(z)
.
\]
The claimed transformation for $P$ then follows from this transformation for $U$ by plugging into the definition of $P$:
\[
P\left(\frac 1z\right)=\frac{U\left(\frac 1z\right)}{\left(1-\frac1z\right)^{e+1}}=\frac{U(z)\cdot z}{\left(1-\frac 1z\right)z^{e+1}}
=
(-1)^{e+1}z\frac{U(z)}{(1-z)^{e+1}}=(-1)^{e+1}zP(z).
\]
\end{proof}

\subsection{Zeros of Period Polynomials}

Extending the work of Conrey, Farmer, and Imamo$\mathrm{\bar{g}}$lu \cite{CFI} and El-Guindy and Raji \cite{ER}, it is now known that period polynomials
of newforms satisfy the Riemann Hypothesis. More precisely, we have the following theorem.

\begin{theorem}[\cite{JinMaOnoSound}, Theorem 1.1]\label{RHPP}
If $f\in S_k(\Gamma_0(N))$ is an even weight $k\geq 4$ newform, 
then all zeros of the period polynomial $R_f(z)$ lie on the unit circle.
\end{theorem}
\begin{remark}
The original result in \cite{JinMaOnoSound} states an equivalent result for a slightly differently normalized polynomial (which involves a rescaling of the variable $z$, and hence a stretching of the circle that the zeros lie on).
\end{remark}
As we shall see, this theorem will provide the link between Theorem~\ref{Thm1} and Theorem~\ref{RVThm}. We also require the following basic result.
\begin{lemma}\label{ValueOneNotZeroPP}
Under the same conditions as Theorem \ref{RHPP}, we have that $R_f(1)\neq0$ if $\epsilon(f)=1$ and $R_f(s)$ has a simple zero at $s=1$ if $\epsilon(f)=-1$.
\end{lemma}
\begin{proof}
The functional equation for $\Lambda(f,s)$ shows that
\begin{equation}\label{RF1Eqn}
R_f(1)
=
\epsilon(f)\sum_{j=0}^{k-2} \binom{k-2}{j}\Lambda(f,j+1)
=\begin{cases}
\binom{k-2}{\frac{k-2}2} \Lambda\left(f,\frac k2\right)+2\sum_{j=\frac k2}^{k-2}\binom{k-2}{j}\Lambda(f,j+1)& \text{ if }\epsilon(f)=1,
 \\
-\binom{k-2}{\frac{k-2}2}\Lambda\left(f,\frac k2\right) & \text{ if } \epsilon(f)=-1.
\end{cases}
\end{equation}
Now $\Lambda(f,s)$ is real-valued on the real line, and
well-known work of Waldspurger \cite{W} implies that $\Lambda\left(f,\frac k2\right)\geq0$. Moreover, Lemma 2.1 of \cite{JinMaOnoSound} states that
\begin{equation}\label{ChainIneqCritVal}
0\leq\Lambda\left(f,\frac k2\right)\leq \Lambda\left(f,\frac k2+1\right)\leq\ldots\Lambda(f,k-1)
\end{equation}
and that $\Lambda\left(f,\frac k2\right)=0$ if $\epsilon(f)=-1$. So, if $\epsilon(f)=1$, then the expression in the first case of \eqref{RF1Eqn} is composed of all non-negative terms, which cannot all vanish as it is impossible for all periods of $f$ to be zero. Hence, in this case, $R_f(1)\neq0$. If $\epsilon(f)=-1$, then as $\Lambda\left(f,\frac k2\right)=0$, we see that $R_f(1)=0$. To see that this zero is simple, note in a similar manner that all terms in $R_f'(1)$ are non-positive, with the last term being $(2-k)\Lambda(f,k-1)$. But this term cannot be zero, as the chain of inequalities in \eqref{ChainIneqCritVal} would then imply that all periods of $f$ are zero.
\end{proof}
\subsection{Proof of Theorem~\ref{Thm2}}

Using Newton's Binomial Theorem, we have
\[
(1-z)^{1-k}=\sum_{n\geq0}\binom{k-2+n}{k-2}z^n
,
\]
and so, letting $j\mapsto k-2-j$ in the sum defining $R_f(z)$, using the functional equation for $\Lambda_f$, and sending $n\mapsto n+j-(k-2)$ gives
\[
\frac{R_f(z)}{(1-z)^{k-1}}
=
\epsilon(f)\sum_{n=0}^{\infty}z^n\sum_{j=0}^{k-2}\binom{k-2}{j}\Lambda(f,j+1)\binom{n+j}{k-2}
.
\]
Calling the coefficient of $z^n$ in this last expression $h_n$, we find that
\begin{equation*}
\begin{aligned}
h_n
&
=
\epsilon(f)\frac1{(k-2)!}\sum_{j=0}^{k-2}\binom{k-2}j\Lambda(f,j+1)\sum_{m=0}^{k-2}s(k-2,m)(n+j)^m
\\
&
=
\epsilon(f)\frac1{(k-2)!}\sum_{j=0}^{k-2}\binom{k-2}j\Lambda(f,j+1)\sum_{m=0}^{k-2}s(k-2,m)\sum_{h=0}^m\binom mh j^{m-h}n^h
\\
&
=
\epsilon(f)\frac{1}{(k-2)!}\sum_{h=0}^{k-2}n^h\sum_{m=h}^{k-2}\binom mh s(k-2,m)\sum_{j=0}^{k-2}\binom{k-2}j\Lambda(f,j+1)j^{m-h}
\\
&
=
\epsilon(f)\frac{1}{(k-2)!}\sum_{h=0}^{k-2}n^h\sum_{m=0}^{k-2-h}\binom{m+h}hs(k-2,m+h)\sum_{j=0}^{k-2}\binom{k-2}j\Lambda(f,j+1)j^m
,
\end{aligned}
\end{equation*}
which is $Z_f(-n)$ by definition.

\subsection{Proof of Theorem~\ref{Thm1}}

We begin by setting
\[
\widehat R_f(z):=\frac{R_f(z)}{(1-z)^{\delta_{-1,\epsilon(f)}}}
,
\]
where $\delta_{i,j}$ is the Kronecker delta function. 
By Theorem \ref{RHPP} and Lemma \ref{ValueOneNotZeroPP}, we see that $\widehat R_f$ is a polynomial of degree $k-2-\delta_{-1,\epsilon(f)}$ all of whose roots lie on the unit circle and such that $\widehat R_f(1)\neq0$.
Thus, we have
\[
\frac{R_f(z)}{(1-z)^{k-1}}=\frac{\widehat R_f(z)}{(1-z)^{k-1-\delta_{-1,\epsilon(f)}}}
.
\]
Applying Theorem \ref{Thm2} and Theorem \ref{RVThm} with $e=k-2-\delta_{-1,\epsilon(f)}$ yields the result, and in particular shows that the zeros of $Z_f(s)$ lie on the line $\re(s)=1/2$.

\subsection{Proof of Theorem~\ref{ZeroDistribution}}\label{ProofZeroDistThm}

\begin{proof}[Proof of Theorem~\ref{ZeroDistribution}]
To prove (1) we note that the polynomials $H^-_{k}(x)$ are Rodriguez-Villegas transforms of $\sum_{j=0}^{k-3} x^j$ and that the $H^+_k(x)$ are the transforms of $x^{k-2}+1$. For example, we have that
$$
\frac{\sum_{j=0}^{k-3} x^j}{(1-x)^{k-2}}=\sum_{n\geq0}H^-_{k}(n)x^n
$$
That  the zeros of $H^{-}_{k}(x)$ are on $\re(x)=1/2$ is one of the main examples of Theorem \ref{RVThm} in \cite{BHW}. The precise location of its zeros in Theorem \ref{ZeroDistribution} is a recapitulation of the statement and proof of Theorem 1.7 of \cite{BHW}. A simple modification of the proof there yields the locations of the zeros for $H^+_k(x)$. Essentially, they are the same as the polynomials considered in (2.7) there, but with the degree $n$ shifted by $1$ and with the minus sign in the last expression replaced by a plus, and their proof is then easily adapted by keeping track of an extra sign throughout. 
 
The proof of part (2) follows directly from part (i) of Theorem 1.2 in \cite{JinMaOnoSound}. Noting that our polynomials $R_f(z)$ are related to
$r_f(z)$ in the notation of \cite{JinMaOnoSound} by
$$
R_f(z)= (\sqrt{N}/i)^{k-1} \cdot r_f\left(\frac{z}{i\sqrt{N}}\right).
$$
Therefore, when $k=4$ and $\epsilon(f)=-1$ we see that $R_f$ is a multiple of $z^2-1$. Hence, its corresponding $Z_f(s)$ is a multiple of $2s-1$. Similarly, when $\epsilon(f)=1$, again by Theorem 1.2 of \cite{JinMaOnoSound} we see that the roots of $R_f(z)$ lie arbitrarily close to $\pm i$ as $N\rightarrow\infty$. Hence, the Rodriguez-Villegas transform becomes arbitrarily close in the limit to the transform of (a multiple of) $z^2+1$, and so the coefficients of $Z_f(s)$ tend to those of (a multiple of) the polynomial $s^2-s+1$. As extracting roots of a polynomial is continuous in the coefficients of the polynomial, we have the desired convergence of the roots of $Z_f(s)$ in the limit.

We will prove part (3) similarly using 
Theorem 1.2 (ii)  \cite{JinMaOnoSound} to determine the zeros of $R_f(z)$ to high accuracy. That is, we can rephrase Theorem 1.2 (ii) of \cite{JinMaOnoSound} as saying that for large $N$, the roots of $R_f(s)$ may be written as 
$$
 \exp\left( i \theta_\ell + O\left(\frac{1}{2^k\sqrt{N}}\right) \right),
$$ 
where for $0\le \ell \le k-3$ we denote by $\theta_\ell$ the unique solution in $[0,2\pi)$ to the equation  
$$ 
\left(\frac{k-2}2\right)\theta_\ell - \frac{2\pi}{\sqrt{N}} \sin \theta_\ell = \begin{cases} 
\frac \pi2 +\ell \pi &\text{if  } \epsilon(f)=1,\\ 
\ell \pi &\text{if  } \epsilon(f)=-1. \\
\end{cases}
$$
Now as $N$ grows, the angles $\theta_\ell$ are very nearly the solutions in $[0,2\pi)$ of the equation 
$$ 
\left(\frac{k-2}2\right)\theta'_\ell = \begin{cases} 
\frac \pi2 +\ell \pi &\text{if  } \epsilon(f)=1\\ 
\ell \pi &\text{if  } \epsilon(f)=-1, \\
\end{cases}
$$
which are exactly the roots of $H^{\pm}_k(z)$ where $\epsilon(f)=\pm1$.
Due to the presence of the $\sqrt{N}$ in the denominator of the error term in the estimation of the roots of a period polynomial $R_f(z)$ above, we conclude that the coefficients of $R_f(z)$ as $N\rightarrow\infty$ are approaching those of a multiple of $z^{k-2}\pm 1$. The result then follows directly from part (1) and the fact that taking roots of polynomials is a continuous operation depending on their coefficients. We note that the matching of distributions of zeros of the two polynomials is made possible by the fact (cf. Lemma \ref{ValueOneNotZeroPP}) that both $z^{k-2}\pm1$ and $R_f(z)$ have the same order of vanishing at $z=1$ and hence that their Rodriguez-Villegas transforms have the same degrees.

To prove the strict upper bound on the imaginary parts of roots, we first consider the case when $\epsilon(f)=1$. Then we directly apply Theorem 1 of \cite{B}, using the positivity properties of critical completed $L$-values reviewed in Lemma \ref{ValueOneNotZeroPP}. When $\epsilon(f)=-1$, by Lemma \ref{ValueOneNotZeroPP} and the functional equation for $\Lambda(f,s)$, we see that $R_f(x)/(1-x)$ is a polynomial with all non-positive coefficients, and so the Rodriguez-Villegas transform of $R_f$ is the same as the Rodriguez-Villegas transform of this polynomial with the degree lowered by $1$. This shows that we may again apply Theorem 1 of \cite{B}, by applying it to the polynomial $R_f(s)/(x-1)$ with positive coefficients and whose transform has the same zeros as $Z_f(s)$.

\end{proof}

\section{Ingredients for Theorem~\ref{Thm3}}\label{BlochKato}

We first describe the local conditions $H^1_\mathbf{f}(\Q_p,V_\lambda(j))$ for a given prime $p$, following \cite[Section 3]{BK}. Recall that $\lambda$ was the prime above $l$ in Deligne's representation $V_\lambda$. 

The first case is when $p=l$. Here, we define $$H^1_\mathbf{f}(\Q_p,V_\lambda(j)):=\ker\left(H^1_\mathbf{f}(D_p,V_\lambda(j))\rightarrow H^1_\mathbf{f}(D_p,V_\lambda(j)\otimes\mathbb{B}_{\text{cris}})\right),$$
where $D_p$ denotes a decomposition group for a prime over $p$. For a definition of the $\Q_p$-algebra $\mathbb{B}_{\text{cris}}$, we refer to Berger's article \cite[II.3]{Berger}.

For the other cases (i.e. $p\neq l$), we let $$H^1_\mathbf{f}(\Q_p,V_\lambda(j)):=\ker\left(H^1_\mathbf{f}(D_p,V_\lambda(j))\rightarrow H^1_\mathbf{f}(I_p,V_\lambda(j))\right),$$
where $I_p$ the inertia subgroup. We let $H^1_\mathbf{f}(\Q,V_\lambda(j))$ be the corresponding global object, i.e. the elements of $H^1(\Q,V_\lambda(j))$ whose restriction at $p$ lies in $H^1_\mathbf{f}(\Q_p,V_\lambda(j))$. 

We note that Bloch and Kato's Tamagawa number conjecture \ref{bkt} is independent of any choices, cf. \cite[Section 6]{D}, or for more detail cf. \cite[Proposition 5.14 (iii)]{BK} and \cite[page 376]{BK}, in which the independence of the choice in lattice in the Betti cohomology is discussed. 

Second, we describe the set of global points $H^0_\Q$, with the appropriate Tate twists:

$$ H^0_\Q(j):=\bigoplus_\lambda H^0(\Q,V_\lambda/T_\lambda(j)).$$

\begin{proof}[Proof of Theorem~\ref{Thm3}]
The proof of Theorem~\ref{Thm3} follows immediately from replacing the terms involving $L(f,j+1)$ by $\widetilde{C(j+1)}$, and appropriate normalizations.
\end{proof}

\section{Examples}\label{ExSec}

We conclude with examples which illustrate the results in this paper.

\subsection{Zeta function for the modular discriminant}

We consider the normalized Hecke eigenform $f=\Delta\in S_{12}(\Gamma_0(1))$. In this case, $\epsilon(f)=1$ and we have
\begin{align*}
R_{\Delta}(z)
\approx\ 
&
0.114379\cdot\left(\frac{36}{691}z^{10}+z^8+3z^6+3z^4+z^2+\frac{36}{691}\right)
\\
&
\ \ \ \ \ \ +0.00926927\cdot(4z^9+25z^7+42z^5+25z^3+4z).
\end{align*}
The ten zeros of $R_{\Delta}$ lie on the unit circle, and are approximated by the set
$$
\left\{
\pm i,-0.465\pm 0.885i,\
 -0.744\pm 0.668i,\
-0.911 \pm0.411i,\
-0.990\pm0.140i
\right\}.
$$
These are illustrated in the following diagram.
\begin{figure}[H]
\centering
\includegraphics[scale=0.5,keepaspectratio=true]{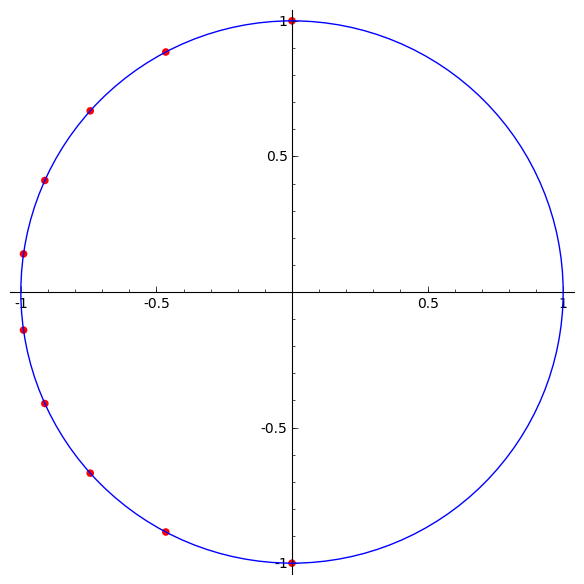}
\caption{The roots of $R_{\Delta}(z)$}
\end{figure}

By taking the Rodriguez-Villegas transform and letting $s\mapsto-s$ we find that 
\begin{align*}
&
Z_{\Delta}(s)\approx
(5.11\times 10^{-7})s^{10}
-(2.554\times 10^{-6})s^9 +(6.01\times 10^{-5})s^8
-( 2.25\times 10^{-4})s^7 
\\
&
+0.00180s^6 
-0.00463s^5+0.0155s^4
-0.0235s^3 +0.0310s^2
-0.0199 s
+
0.00596
  .
\end{align*} 
Theorem~\ref{Thm1} establishes that its zeros $\rho$ satisfy $\re(\rho)=1/2$; indeed, they are  approximately
\begin{align*}
\left\{
\frac12\pm 8.447i,
 \frac12\pm 5.002i,
\frac12 \pm 2.846i,
\frac12 \pm 1.352i,
\frac12 \pm 0.349i,
\right\}
,
\end{align*}
as illustrated in the next figure.
\begin{figure}[H]
\centering
\includegraphics[scale=0.5,keepaspectratio=true]{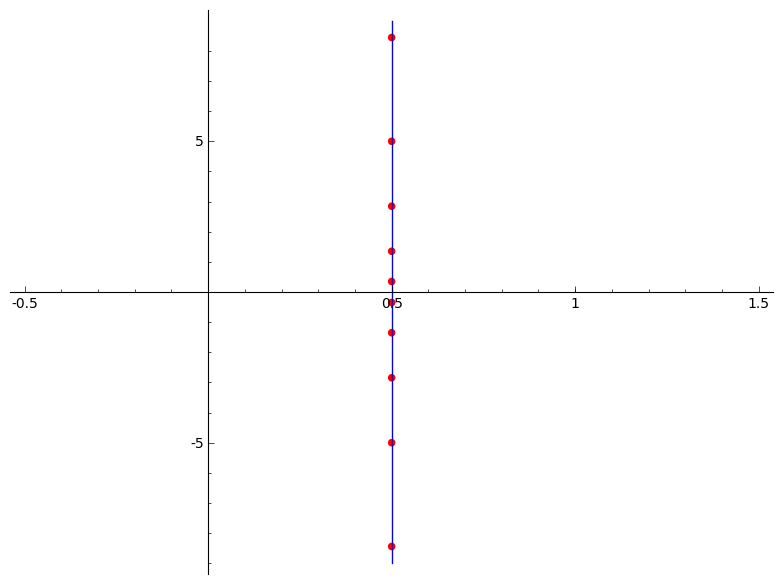}
\caption{The roots of $Z_{\Delta}(s)$}
\end{figure}

\subsection{Ehrhart polynomials and newforms of weight $6$}

Here we consider newforms $f\in S_6(\Gamma_0(N))$ with $\epsilon(f)=-1$.  
By Theorem~\ref{ZeroDistribution} (2), 
the roots of $Z_f(s)$ are closely related to the roots of the Ehrhart polynomial of the convex hull 
\[
\operatorname{conv}\left\{e_1,e_2,e_3,-e_1-e_2-e_3\right\}
.
\]  
The following image renders this tetrahedron. 
\begin{figure}[H]\label{FigureTetra}
\centering
\includegraphics[scale=0.5,keepaspectratio=true]{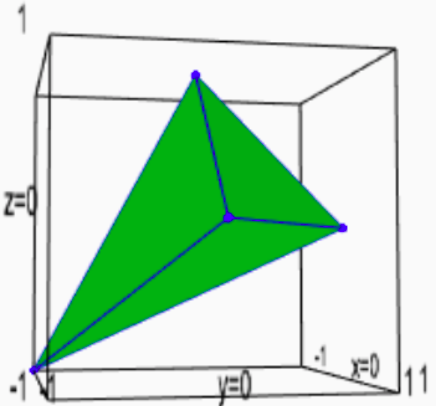}
\caption{The tetrahedron whose Ehrhart polynomial is $H_6^-(s)$.}
\end{figure}

The corresponding Ehrart polynomial counts the number of integer points in dilations of Figure \ref{FigureTetra}, and is given by the Rodriguez-Villegas transform of 
$1+x+x^2+x^3$. Namely, we have
\[
H_6^-(s)=
\binom{s+3}{3}+\binom{s+2}{3}+\binom{s+1}{3}+\binom{s}{3}
=\frac23s^3 + s^2 + \frac73s + 1
.
\]
Therefore, we find that
$$
\lim_{N\rightarrow +\infty} \widetilde{Z}_f(s)=\widetilde{H}^{-}_6(-s)= \left(s-\frac{1}{2}\right)
\left(s-\frac{1}{2}+\frac{\sqrt{-11}}{2}\right) \left(s-\frac{1}{2}-\frac{\sqrt{-11}}{2}\right),
$$
where the limit is over newforms $f\in S_6(\Gamma_0(N))$ with $\epsilon(f)=-1$, and where the polynomials have been normalized to have leading coefficient 1.


\end{document}